\documentclass[11pt]{article}

\usepackage{jhtitle}
\usepackage{geometry}
\usepackage{amsbsy,amsmath,amsthm,amssymb,graphicx}
\usepackage{natbib}
\usepackage{setspace}
\usepackage{subcaption}

\newcommand{\pcite}[1]{\citeauthor{#1}'s \citeyearpar{#1}}

\newcommand{\df}{\mathrm{d}}
\def\baro{\vskip  .2truecm\hfill \hrule height.5pt \vskip  .2truecm}
\def\barba{\vskip -.1truecm\hfill \hrule height.5pt \vskip .4truecm}

\newtheorem{theorem}{Theorem}[section]
\newtheorem{lemma}[theorem]{Lemma}
\newtheorem{corollary}[theorem]{Corollary}
\newtheorem{remark}[theorem]{Remark}
\newtheorem{proposition}[theorem]{Proposition}

\newcounter{cond}

\newcommand{\X}{{\mathsf{X}}}
\newcommand{\Y}{{\mathsf{Y}}}

\title{Geometric convergence bounds for Markov chains in Wasserstein
  distance based on generalized drift and contraction conditions}

\author{Qian Qin \\ School of Statistics \\ University of Minnesota \and James P. Hobert \\ Department of Statistics
  \\ University of Florida} 

\date{\today} 


\begin{document}
	
\maketitle
	
\begin{abstract}
  Let $(X_n)_{n=0}^\infty$ denote a Markov chain on a
  Polish space that has a stationary distribution~$\varpi$.  This
  article concerns upper bounds on the Wasserstein distance
  between the distribution of $X_n$ and~$\varpi$.  In particular, an
  explicit geometric bound on the distance to stationarity is derived
  using generalized drift and contraction conditions whose parameters
  vary across the state space.  These new types of drift and
  contraction allow for sharper convergence bounds than the standard
  versions, whose parameters are constant.  Application of the result
  is illustrated in the context of a non-linear autoregressive
  process and a Gibbs algorithm for a random effects model.
\end{abstract}
	
\section{Introduction}
\label{dec:intro}

The study of convergence to stationarity of Markov chains commonly
requires the specification of a metric on an appropriate space of
probability distributions.  The standard has long been total variation
(TV) distance, but, over the last decade or so, Wasserstein distance
has received a good deal of attention as well.  One obvious reason for
studying convergence in Wasserstein distance is that there exist
Markov chains that do not actually converge in TV distance, but do
converge in some Wasserstein distance \citep[see,
  e.g.,][]{hairer2011asymptotic,butkovsky2014subgeometric}.  Another, perhaps more
important, reason stems from the current focus on so-called big data
problems, which leads to the study of Markov chains on
high-dimensional state spaces.  Indeed, it is becoming clear that the
techniques used for developing Wasserstein bounds are more robust to
increasing dimension than are those used to construct TV bounds
\citep[see,
  e.g.,][]{hairer2011asymptotic,hairer2014spectral,durmus2015quantitative,eberle2019mckean,qin2020wasserstein}.
In this paper, we study convergence rates of Markov chains with
respect to Wasserstein distances.  In particular, we develop explicit
geometric (exponential) convergence bounds using generalized (or localized)
versions of the usual drift and contraction conditions.

Previous work devoted to convergence analysis of Markov chains with
respect to Wasserstein distances includes \cite{jarner2001locally},
\cite{hairer2011asymptotic}, \cite{butkovsky2014subgeometric},
\cite{durmus2015quantitative}, and \cite{douc2018markov}.  
A recurring theme in these papers is
the combination of drift and contraction.  The basic
program is to first establish a strong contraction condition on a
coupling set (which is a subset of $\X \times \X$, where $\X$ is the
state space), and then to use a Lyapunov drift condition to drive a
coupled version of the Markov chain towards that subset. 
The goal of this article is to provide a more refined version of this program, where instead of considering only two sets of points in $\X \times \X$, i.e., the coupling set and its complement, one takes into account the localized behavior of the Markov chain at each point of the product state space.
Indeed, we demonstrate that geometric convergence bounds can be constructed using a generalized version of drift and contraction, which we now describe.

When developing drift and contraction conditions for specific
problems, the parameters in these inequalities are often initially
non-constant.  The varying parameters may encode rich information
about the dynamics of the chain.  The usual drift and contraction
conditions (with constant parameters) are typically obtained by taking
the supremum over the coupling set, and again over
the compliment of the coupling set.  Naturally, this process can
result in a substantial loss of information.  The bounds that we
provide can be constructed directly from the drift and
contraction conditions with non-constant parameters --- the
generalized drift and contraction conditions.  This procedure does not
require selecting a coupling set, and can potentially lead to sharper
bounds based on weaker assumptions, compared to previous results.  For
example, our upper bound on the geometric convergence rate is always
better (i.e., smaller) than that in \cite{durmus2015quantitative}.
Our work draws inspiration from the ``small function'' version of the
minorization condition \citep[see, e.g.,][Section~2.3]{numm:1984},
which can be considered a generalized version of the usual
minorization condition (with constant parameter).  The local
contractive behavior of a Markov chain considered in
\cite{steinsaltz1999locally} and \cite{eberle2019quantitative} is also
related to the generalized conditions that we use.

The rest of the article is organized as follows.  In
Section~\ref{sec:general}, after setting notation, we give a detailed
account of generalized drift and contraction conditions.  We construct
a geometric convergence bound based on these conditions, and compare
it to bounds based on standard drift and contraction.  The proofs are
postponed to Section~\ref{sec:proof}.  In Section~\ref{sec:example},
we use a perturbed autoregressive process and a Gibbs algorithm for a random effects model to demonstrate how our
results can be applied.  These applications provide concrete examples
of the extent to which the new bound improves upon standard ones.
Finally, Appendix~\ref{app:compare} contains a rigorous comparison of
the bounds developed herein with those of
\citet{durmus2015quantitative}, and Appendix~\ref{app:optimize}
contains some technical details supporting the analysis in
Section~\ref{sec:example}.

\section{Generalized Drift and Contraction}
\label{sec:general}
	
%Let $(\X, \mathcal{B})$ be a countably generated measurable space such
%that each singleton in~$\X$ is measurable.  Let $\varrho: \X \times \X
%\to [0,\infty)$ be a measurable metric.
Let  $(\X,\varrho)$ be a complete separable metric space (i.e., Polish metric space), and assume that~$\mathcal{B}$
is the associated Borel $\sigma$-algebra. 
%When we assume that $(\X,
%  \varrho, \mathcal{B})$ is a Polish metric space, we mean that $(\X,
%  \varrho)$ is a complete separable metric space, and that~$\mathcal{B}$
%  is the associated Borel algebra.  
%  When Polish-ness is not assumed,
%  we make no explicit assumption about the relationship
%  between~$\varrho$ and~$\mathcal{B}$.  
  Let $\mathcal{P}(\X)$ denote the
  set of probability measures on $(\X, \mathcal{B})$, and let
  $\delta_x$ denote the point mass (or Dirac measure) at $x$.  For
  $\mu, \nu \in \mathcal{P}(\X)$, let
\[
\mathcal{C}(\mu, \nu) = \big \{ \upsilon \in \mathcal{P}(\X \times
\X): \; \upsilon(A_1 \times \X) = \mu(A_1) \,, \; \upsilon(\X \times
A_2) = \nu(A_2) \text{ for all } A_1,A_2 \in \mathcal{B} \big \} \,.
\]
This is the set of \textit{couplings} of $\mu$ and $\nu$.  
Let $\psi: \X \times \X \to [0,\infty)$ be a lower semi-continuous function.
This function will be referred to as a cost function, and is usually taken to be a distance-like function that is vanishing if and only if its two arguments coincide.
The
Wasserstein divergence between $\mu \in \mathcal{P}(\X)$ and $\nu \in \mathcal{P}(\X)$ induced by~$\psi$ is defined to be
\[
W_{\psi}(\mu, \nu) = \inf_{\upsilon \in \mathcal{C}(\mu,\nu)} \int_{\X
  \times \X} \psi(x,y) \upsilon(\df x, \df y) \,.
\]
Taking $\psi = \varrho$ gives the $L_1$-Wasserstein (or Kantorovich-Rubinstein) {\it distance} induced by~$\varrho$.
In fact, for $p \geq 1$, one can define
\[
\mathcal{P}_{\varrho,p}(\X) = \left\{ \mu \in \mathcal{P}(\X) \,: \;
\int_{\X} \varrho(x, y)^p \, \mu(\df y) < \infty \text{ for some } x \in \X
\right\} ,
\]
and let
\[
W_{\varrho,p} (\mu, \nu) = \left( \inf_{\upsilon \in \mathcal{C}(\mu,\nu)} \int_{\X
	\times \X} \varrho(x,y)^p \, \upsilon(\df x, \df y) \right)^{1/p} .
\]
Then $W_{\varrho,p}$ is the $L_p$-Wasserstein distance, and $(\mathcal{P}_{\varrho,p}(\X), W_{\varrho,p})$ forms a Polish metric space \citep[see, e.g.,][Theorem 6.18]{villani2008optimal}.
%Define 
%\[
%\mathcal{P}_{\varrho}(\X) = \left\{ \mu \in \mathcal{P}(\X) \,: \;
%\int_{\X} \varrho(x, y) \mu(\df y) < \infty \text{ for some } x \in \X
%\right\} \,.
%\]
%Taking $\psi = \varrho$ in $W_{\psi}$ gives a genuine metric on $\mathcal{P}_{\varrho}(\X)$, and
%$(\mathcal{P}_{\varrho}(\X), W_{\varrho})$ is itself a Polish metric space
%\citep[see, e.g.,][Definition 6.1 and Theorem
%  6.18]{villani2008optimal}.  $W_{\varrho}$ is called the
%$L_1$-Wasserstein (or Kantorovich-Rubinstein) \textit{distance}
%induced by~$\varrho$.

Let $P: \X \times \mathcal{B} \to [0,1]$ be a Markov transition kernel
(Mtk), and for a positive integer~$n$, let $P^n: \X \times \mathcal{B} \to
[0,1]$ be the corresponding $n$-step Mtk.  (As usual, we write $P$
instead of $P^1$.)  Let $P^0$ be the identity kernel, i.e., $P^0(x,A)
= 1_{x \in A}$ for $A \in \mathcal{B}$.  For $\mu \in \mathcal{P}(\X)$ and a measurable
function $f: \X \to \mathbb{R}$, let $\mu f = \int_{\X} f(x) \mu(\df
x)$, $\mu P^n (\cdot) = \int_{\X} P^n(x,\cdot) \mu(\df x)$, and $P^n f
(\cdot) = \int_{\X} f(x) P^n(\cdot, \df x)$ (assuming the integrals
are well-defined).

Our goal is to establish new conditions under which the Markov chain
defined by $P$ converges in $W_{\psi}$ to a limiting distribution $\varpi
\in \mathcal{P}(\X)$ at a geometric rate.  More specifically, based on these new conditions, we will
construct convergence bounds of the form
\[
W_{\psi}(\delta_x P^n, \varpi) \leq c_x \, \rho^n \,,
\]
where $x \in \X$, $n \in \{0,1,2,\dots\} =:
\mathbb{Z}_+$, $c_x < \infty$, and $\rho < 1$.  
Moreover, we will
provide explicit formulas for $c_x$ and $\rho$.  
$\rho$ is an upper bound on the geometric convergence rate of the chain, defined as
\[
\exp \left( \sup_{x \in \X} \limsup_{n \to \infty} \frac{\log W_{\psi}(\delta_x^n, \varpi)}{n} \right)
\] 
\citep[see][page 40]{roberts2001geometric}.
We will be particularly interested in finding smaller values of~$\rho$ compared to existing bounds in the literature.

We first review a set of standard drift and contraction conditions that can be used to produce this type of bound.
\begin{enumerate}
	\item [$(A1)$] There exist a measurable function $V: \X \to
	[0,\infty)$ and $a \in [0,\infty)$ such that
	\begin{equation}
	\label{ine:Vpsi}
	\psi(x,y) \leq a [V(x) + V(y) + 1] \,, \quad (x,y) \in \X
	\times \X \,,
	\end{equation}
	and there exist $\eta \in [0,1)$ and $L \in [0,\infty)$ such that
	\begin{equation}
	\label{ine:drift}
	PV(x) \leq \eta V(x) + L \,, \quad x \in \X \,.
	\end{equation}
	\item [$(A2)$] There exist $\gamma \in [0,1)$ and $K \in [0,\infty)$ such that
	for each $(x,y) \in \X \times \X$,
	\[
	W_{\psi}(\delta_x P, \delta_y P) \leq \begin{cases} \gamma \psi(x,y)
	\,, & (x,y) \in C \,,\\ K \psi(x,y) \,, & (x,y) \not\in C \,,
	\end{cases}
	\]
	where the coupling set~$C$ is defined to be $\big \{ (x,y) \in \X
	\times \X: V(x) + V(y) < d \big \}$, with $d > 2L/(1-\eta)$.
\end{enumerate}
Condition $(A2)$, referred to as a contraction condition, partitions
the product state space $\X \times \X$ into two regions,~$C$ and its
complement.  The coupling set~$C$ is a ``good" set, where two coupled
copies of the Markov chain tend to approach each
other.  Condition $(A1)$, referred to as a drift condition, is a
classical assumption that drives the coupled chains towards~$C$.
Taking into account the magnitude of contraction and drift (quantified
by parameters such as~$\gamma$,~$L$, and~$K$), it's possible to
construct a quantitative convergence bound.  The following is an
example of this.

\begin{corollary}
  \label{cor:geometric}  
  Suppose that $(A1)$, $(A2)$, and the following condition all hold:
\begin{enumerate}
  \item[$(A3)$] Either $K \leq 1$ or
\begin{equation}
  \label{ine:KA1}
  \log K \log (2L+1) < \log \frac{1}{\gamma} \log \frac{d + 1}{\eta d
    + 2L + 1} \,.
\end{equation}
\end{enumerate}
Then, for each $\mu,\nu \in \mathcal{P}(\X)$ and real number~$r$ such that
\begin{equation}
  \label{ine:r}
  \frac{\log(2L+1)}{\log(2L+1) + \log (1/\gamma)} < r < \frac{\log [
       (d+1) / (\eta d + 2L + 1) ]}{\log (K \vee 1) + \log [  (d+1) / (\eta d +
       2L + 1) ]} \,,
\end{equation}
we have
\begin{equation}
  \nonumber
  W_{\psi}(\mu P^n, \nu P^n) \leq a (\mu V + \nu V + 1) \rho_r^n \,, \quad n \in \mathbb{Z}_+,
\end{equation} 
where
\begin{equation}
  \nonumber
  \rho_r = \left[ \gamma^r (2L+1)^{1-r} \right] \vee \left[ K^r \left(
    \frac{\eta d + 2L + 1}{d + 1} \right)^{1-r} \right] < 1 \,.
\end{equation}
In particular, if~$P$ admits a stationary distribution~$\varpi$, then for $\mu \in \mathcal{P}(\X)$,
\[
W_{\psi}(\mu P^n,\varpi) \leq a \left( \mu V + \frac{L}{1-\eta} + 1 \right) \rho_r^n \,, \quad n \in \mathbb{Z}_+.
\]
\end{corollary}

\begin{remark}
The fact that $d > 2L / (1-\eta)$ implies that $(\eta d + 2L +
1)/(d+1) < 1$, and it follows that
\[
0 \leq \frac{\log(2L+1)}{\log(2L+1) + \log (1/\gamma)} < \frac{\log [
	(d+1) / (\eta d + 2L + 1) ]}{\log (K \vee 1) + \log [  (d+1) / (\eta d +
	2L + 1) ]} \leq 1
\]
whenever $K \leq 1$ or~\eqref{ine:KA1} is satisfied.
Thus,~\eqref{ine:r} makes sense, and $r \in (0,1)$.
\end{remark}
\begin{remark} \label{rem:stationary}
We provide conditions for the existence and uniqueness of stationary distributions in a more general context in Propositions~\ref{pro:limiting} and~\ref{pro:stationary}.
Based on these propositions, one can check that, if $(A1)$-$(A3)$ hold and $\psi = \varrho$, then~$P$ admits a unique stationary distribution~$\varpi$.
\end{remark}
\begin{remark}
	One can use Proposition~\ref{pro:limiting} to derive the following alternative convergence bound when $(A1)$-$(A3)$ hold and $\psi = \varrho$.
	For $x \in \X$ and~$r$ satisfying~\eqref{ine:r},
	\[
	W_{\varrho}(\delta_x P^n, \varpi) \leq a \frac{(\eta+1)V(x) + L + 1}{1 - \rho_r} \rho_r^n \,, \quad n \in \mathbb{Z}_+\,,
	\]
	where $\rho_r \in [0,1)$ is given in Corollary~\ref{cor:geometric}.
	However, this bound is slightly looser than the one in Corollary~\ref{cor:geometric}.
\end{remark}
	
	Corollary~\ref{cor:geometric} can be derived using the main theorem of this article, which will soon be stated.
	Results similar to Corollary~\ref{cor:geometric} have been derived in various contexts; see, e.g., \cite{butkovsky2014subgeometric} and \cite{durmus2015quantitative}. 
	The most recent account is Theorem~20.4.5 of \cite{douc2018markov}, where the authors establish geometric convergence under a condition akin to $(A1)$ along with $(A2)$ with $K \leq 1$.
	\pcite{douc2018markov} Theorem~20.4.5 does not provide a fully explicit convergence bound, although it is possible to derive such a bound based on the proof of said result, and it will resemble what is given in Corollary~\ref{cor:geometric}.
	A fully computable geometric convergence bound is given in \cite{durmus2015quantitative} for the case that $\psi=\varrho$,~$\varrho$ is bounded, and $K \leq 1$.
	The bound therein is similar to what is given in Corollary~\ref{cor:geometric}, albeit slightly looser, as shown in Section~\ref{app:compare} of the Appendix.

	One can also derive the following continuous version of Corollary~\ref{cor:geometric}, whose proof will be given in Section~\ref{sec:proof}.
	\begin{corollary} \label{cor:continuous}
		Let $(P^t)_{t \geq 0}$ be a Markov semigroup on $(\X,\mathcal{B})$.
		Suppose that there exists $t_* > 0$ such that
		$P^{t_*}$ (in place of~$P$) satisfies $(A1)$-$(A3)$, and that~$P^{t_*}$ admits a unique stationary distribution~$\varpi$.
		Suppose further that
		there exists $b < \infty$ such that for every $(x,y) \in \X \times \X$ and $t \in [0,t_*)$,
		\begin{equation} \label{ine:smallt}
		W_{\psi}(\delta_x P^t, \delta_y P^t) \leq b\psi(x,y) \;.
		\end{equation}
		Then~$\varpi$ is also the unique stationary distribution of $(P^t)_t$.
		Moreover, for $\mu \in \mathcal{P}(\X)$ and $t \geq 0$,
		\[
		W_{\psi}(\mu P^t, \varpi) \leq ab \left( \mu V + \frac{L}{1-\eta} + 1 \right) \rho_r^{\lfloor
			t/t_* \rfloor} \;,
		\]
		where $\lfloor \cdot \rfloor$ returns the largest
		integer that does not exceed its argument,~$r$ satisfies~\eqref{ine:r}, and $\rho_r
		\in [0,1)$ is defined as in Corollary~\ref{cor:geometric}.
	\end{corollary}

	Looking back at Condition~$(A2)$, a natural question that can be raised is, whether a more delicate analysis is possible if one divides the product space $\X \times \X$ into more than two parts.
	To be precise, instead of categorizing the behavior of $W_{\psi}(\delta_x P, \delta_y P)$ into two cases, can one gain more by taking into account the local behavior of $W_{\psi}(\delta_x P, \delta_y P)$ for each value of $(x,y)$?
	This prompts us to study the convergence properties of a Markov chain under the following generalized versions
	of $(A1)$ and $(A2)$.
	\begin{enumerate}
		\item [$(B1)$] There exists a measurable function $V:
		\X \to [0,\infty)$ and $a \in [0,\infty)$ such that
		\[
		\psi(x,y) \leq a [V(x) + V(y) + 1] \,, \quad (x,y)
		\in \X \times \X \tag{\ref{ine:Vpsi}} \,,
		\]
		and $PV(x) < \infty$ for each $x \in \X$.
		\item [$(B2)$] There exists a measurable function
		$\Gamma: \X \times \X \to [0,\infty)$ such that for
		each $(x,y) \in \X \times \X$,
		\[
		W_{\psi}(\delta_x P, \delta_y P) \leq \Gamma(x,y) \,
		\psi(x,y) \,.
		\]
	\end{enumerate}
%	Note that by $(B1)$, $W_{\psi}(\delta_x P, \delta_y P) < \infty$ for each $(x,y) \in \X \times \X$, and thus, $(B2)$ holds trivially with
%	\[
%	\Gamma(x,y) = \frac{W_{\psi}(\delta_x P, \delta_y
%		P)}{\psi(x,y)} 1_{x \neq y} \,,
%	\]
%	assuming that the right-hand-side is measurable (which it is, as indicated by Lemma~\ref{lem:kernel}).
%	Of course, in practice, we would like to find a~$\Gamma$ that
%	yields a sharp contraction inequality, but is also simple and
%	well-behaved.  
%	See
%	Section~\ref{sec:example} for concrete examples.
	The relationship between the two conditions above and their standard counterparts is quite obvious.
	In particular, $(A2)$ is just $(B2)$ when $\Gamma$ is constant over some coupling set~$C$ as well as over its complement, assuming that $\Gamma(x,y) < 1$ when $(x,y) \in C$.
	$(B2)$, while being weaker than $(A2)$, may also incorporate information on the localized contractive behavior of~$P$ that $(A2)$ ignores (depending on how~$\Gamma$ is constructed). 
	
	Combining $(B1)$ and $(B2)$ with an analog of $(A3)$ that regulates the relationship between
	$\Gamma(x,y)$ and $(PV(x), PV(y))$ yields our main theorem, which we now state.

	\begin{theorem} \label{thm:main}
		Assume that $(B1)$ and $(B2)$ hold.
		Let $\Lambda: \X \times \X \to [0,\infty)$ be such that
		\[
		\Lambda(x,y) \geq \frac{PV(x) + PV(y) + 1}{V(x) + V(y)
			+ 1} \,.
		\]
		Assume further that the following condition holds:
		\begin{enumerate}
			\item [$(B3)$] There exists $r \in (0,1)$ such that
			\begin{equation} \label{ine:r2}
			\sup_{(x,y) \in \X \times \X } \Gamma(x,y)^r \Lambda(x,y)^{1-r} < 1 \,.
			\end{equation}
		\end{enumerate}
		Then, for $\mu, \nu \in \mathcal{P}(\X)$,
		\[
		W_{\psi}(\mu P^n, \nu P^n) \leq a (\mu V + \nu V+1) \rho_r^n \,, \quad n
		\in \mathbb{Z}_+ \,,
		\]
		where $\rho_r = \sup_{(x,y) \in \X \times \X}
		\Gamma(x,y)^r \Lambda(x,y)^{1-r} < 1$. 
		In particular, if~$P$ admits a stationary distribution~$\varpi$, then
		\begin{equation} \label{ine:geoemtric-pi}
		W_{\psi}(\mu P^n, \varpi) \leq a (\mu V + \varpi V+1) \rho_r^n \,, \quad n
		\in \mathbb{Z}_+ \,.
		\end{equation}
	\end{theorem}
	\begin{remark}
		It can be shown that $(B3)$ holds whenever
		\begin{equation} \nonumber
		\sup_{(x,y) \in \X \times \X} [\Gamma(x,y) \wedge
		\Lambda(x,y)] < 1 \,,
		\end{equation}
		and
		\begin{equation} \nonumber
		\left( \sup_{x,y: \, \Lambda(x,y) > \Gamma(x,y)}
		\frac{\log \Lambda(x,y)}{\log \Lambda(x,y) - \log
			\Gamma(x,y)} \right) \vee 0 < \left( \inf_{x,y: \,
			\Lambda(x,y) < \Gamma(x,y)} \frac{\log
			\Lambda(x,y)}{\log \Lambda(x,y) - \log \Gamma(x,y)}
		\right) \wedge 1 \,.
		\end{equation}
		In this case, $\rho_r < 1$ whenever
		\begin{equation} \nonumber
		\left( \sup_{x,y: \, \Lambda(x,y) > \Gamma(x,y)}
		\frac{\log \Lambda(x,y)}{\log \Lambda(x,y) - \log
			\Gamma(x,y)} \right) \vee 0 < r < \left( \inf_{x,y:
			\, \Lambda(x,y) < \Gamma(x,y)} \frac{\log
			\Lambda(x,y)}{\log \Lambda(x,y) - \log \Gamma(x,y)}
		\right) \wedge 1 \,.
		\end{equation}
	\end{remark}
	
	\begin{remark} \label{rem:piV}
		The bound in~\eqref{ine:geoemtric-pi} is nontrivial only if $\varpi V < \infty$.
		This holds under $(A1)$.
		Indeed, under $(A1)$, $\varpi V \leq L/(1-\eta)$ \citep[][Proposition 4.24]{hairer2018ergodic}.
		If~$\psi$ satisfies additional conditions, e.g., $\psi = \varrho$, then an alternative bound holds without the assumption $\varpi V < \infty$.
		See Propositions~\ref{pro:limiting} and~\ref{pro:stationary}.
	\end{remark}

	The next two propositions give conditions for the existence and uniqueness of stationary distributions.
	
	\begin{proposition} \label{pro:limiting}
		Suppose that $(B1)$-$(B3)$ are satisfied, and that the following also holds:
		\begin{enumerate}
			\item [$(B4)$] There exists $p \in [1,\infty)$ such that
			\[
			\varrho(x,y)^p \leq \psi(x,y), \quad (x,y) \in \X \times \X.
			\]
		\end{enumerate}
		Then there exists $\varpi \in \mathcal{P}_{\varrho,p}(\X)$ and $r \in (0,1)$ such that, for $x \in \X$,
			\begin{equation} \label{ine:geometric}
			W_{\varrho,p}(\delta_x P^n, \varpi) \leq a^{1/p} [PV(x) + V(x) +1]^{1/p} \frac{\rho_r^{n/p}}{1 - \rho_r^{1/p}} \,, \quad n \in \mathbb{Z}_+\,,
			\end{equation}
		where $\rho_r \in [0,1)$ is given in Theorem~\ref{thm:main}.
		Moreover,~$P$ admits at most one stationary distribution.
	\end{proposition}
	
	The condition (B4) has
	been used previously as a condition to guarantee the existence of a
	limiting distribution 
	\citep[][Theorem 20.2.1]{douc2018markov}.
	It may be difficult for $(B4)$ to hold when~$\varrho$ is unbounded.
	To circumvent this, one can replace~$\varrho$ with a bounded metric that is topologically equivalent, such as $\varrho \wedge 1$, in the initial setup of the Polish metric space.
	
	Let $C_b(\X)$ be the set of bounded, continuous real-valued
        functions on~$\X$.  
%        We say that a sequence $(\mu_n)_n
%        \subset \mathcal{P}(\X)$ converges weakly to $\mu \in
%        \mathcal{P}(\X)$ if $\lim\limits_{n \to \infty} \mu_n f = \mu
%        f$ for every $f \in C_b(\X)$.  
        We say that~$P$ is Feller if $Pf \in
        C_b(\X)$ for each $f \in C_b(\X)$.  
	
	\begin{proposition} \label{pro:stationary}
		Suppose that $(B1)$-$(B4)$ hold.
		Suppose further that either of the following conditions holds.
		\begin{enumerate}
			\item [$(B5)$] There exists $K' < \infty$ such that, for $(x,y) \in \X \times \X$, $W_{\varrho}(\delta_x P, \delta_y P) \leq K' \varrho(x,y)$.
			\item [$(B6)$] $P$ is
			Feller.
		\end{enumerate}
		Then~$\varpi$, as defined in Proposition~\ref{pro:limiting}, is
                  the unique stationary distribution of~$P$.
	\end{proposition}
	
	\begin{remark}
		Note that if $(B2)$ holds with a bounded function~$\Gamma$ and $\psi = \varrho$, then $(B4)$ and $(B5)$ hold.
	\end{remark}

	Before going on to the next section, which contains the proofs of Theorem~\ref{thm:main} and subsequent propositions, we  compare Theorem~\ref{thm:main} with Corollary~\ref{cor:geometric}, a result based on standard versions of drift and contraction.
	As we will demonstrate in Section~\ref{sec:proof}, the latter is essentially an application of the former when~$\Lambda$ and~$\Gamma$ are constant over a coupling set,~$C$, as well as over its complement, $(\X \times \X)\setminus C$.
	To make a comparison between the two results, consider the following scenario.
	Suppose that~$P$ satisfies $(A1)$ with a drift function $V: \X \to [0,\infty)$, $a \in [0,\infty)$, $\eta \in [0,1)$, and $L \in [0,\infty)$.
	Then~$P$ satisfies $(B1)$ with the same~$V$ and~$a$.
	Suppose that~$P$ also satisfies $(B2)$ with $\Gamma: \X \times \X \to [0,\infty)$.
	Let
	\[
	\Lambda(x,y) = \frac{\eta V(x) + \eta V(y) + 2L + 1}{V(x) + V(y) + 1} \geq \frac{PV(x) + PV(y) + 1}{V(x) + V(y) + 1}
	\]
	for $(x,y) \in \X \times \X$.
	Assume further that $(B3)$ holds and that~$P$ has a stationary distribution~$\varpi$.
	We know from Remark~\ref{rem:piV} that $\varpi V < \infty$.
	In accordance with Theorem~\ref{thm:main}, for $r \in (0,1)$, let
	\[
	\rho_r^B = \sup_{(x,y) \in \X\times \X} \Gamma(x,y)^r \Lambda(x,y)^{1-r}\,.
	\]
	Whenever $\rho_r^B < 1$, this is a nontrivial upper bound on the chain's geometric convergence rate.
	Here, we use the superscript ``$B$" to indicate that $\rho_r$ is constructed based on $(B1)$-$(B3)$.
	
	Suppose now that we are ignorant of Theorem~\ref{thm:main}, and wish to find a convergence bound using the more standard Corollary~\ref{cor:geometric}.
	We need to convert $(B2)$ to $(A2)$ by letting $\gamma = \sup_{(x,y) \in C} \Gamma(x,y)$, and $K = \sup_{(x,y) \not\in C} \Gamma(x,y)$, where $C = \{(x,y): V(x)+V(y)<d\}$ is a coupling set, and $d > 2L/(1-\eta)$.
	Of course, we need to further assume that $\gamma < 1$ and $K < \infty$, so that $(A2)$ is satisfied.
	For $r \in (0,1)$,
	\[
	\begin{aligned}
	\rho_r^A &:= 
	\left[ \gamma^r (2L+1)^{1-r} \right] \vee \left[ K^r \left(
	\frac{\eta d + 2L + 1}{d + 1} \right)^{1-r} \right] \\
	&\geq
	\left\{ \left[ \sup_{(x,y) \in C} \Gamma(x,y)^r \right] \left[ \sup_{(x,y) \in C} \Lambda(x,y)^{1-r} \right] \right\} \vee \left\{ \left[ \sup_{(x,y) \not\in C} \Gamma(x,y)^r \right] \left[ \sup_{(x,y) \not\in C} \Lambda(x,y)^{1-r} \right] \right\} \,.
	\end{aligned}
	\]
	When $(A3)$ and~\eqref{ine:r} hold, $\rho_r^A < 1$, and $\rho_r^A$ is an upper bound on the chain's convergence rate constructed via Corollary~\ref{cor:geometric}.
	On the other hand, note that
	\[
	\rho_r^B = \left[ \sup_{(x,y) \in C} \Gamma(x,y)^r  \Lambda(x,y)^{1-r} \right]  \vee \left[ \sup_{(x,y) \not\in C} \Gamma(x,y)^r  \Lambda(x,y)^{1-r} \right] \,.
	\]
	It's clear that $\rho_r^B \le \rho_r^A$ for each $r \in (0,1)$.
	So $\rho_r^A < 1$ only if $\rho_r^B < 1$.
	Moreover, the inequality between $\rho_r^A$ and $\rho_r^B$ will
	be strict unless~$\Gamma$ and~$\Lambda$ are related in a very specific
	manner that seems unlikely to hold in practice.
	Thus, Theorem~\ref{thm:main} provides a sharper convergence rate bound than Corollary~\ref{cor:geometric}.
	Such improvement is illustrated in Section~\ref{sec:example}, where we study the behavior of a perturbed autoregressive Markov chain and a Gibbs chain for a random effects model.

	\section{Proofs} \label{sec:proof}

	We first introduce the notion of Markovian coupling kernels.
	Recall that $(\X,\varrho)$ is a Polish metric space,~$\mathcal{B}$ is the associated Borel algebra, and~$\psi$ is a lower semi-continuous cost function.
        Suppose that $P_1$ and $P_2$ are Mtks on $(\X, \mathcal{B})$.
        We say that $\tilde{P}: (\X \times \X) \times (\mathcal{B}
        \times \mathcal{B}) \to [0,1]$ is a (Markovian) coupling
        kernel of $P_1$ and $P_2$ if~$\tilde{P}$ is an Mtk such that
        for each $(x,y) \in \X \times \X$, $\delta_{(x,y)}\tilde{P}$ is in
        $\mathcal{C}(\delta_x P_1, \delta_y P_2)$.  In the special
        case that $P_1 = P_2 = P$, we simply say that~$\tilde{P}$ is a
        coupling kernel of~$P$.  It's obvious that $(B2)$ holds if
        there exists a coupling kernel of~$P$, denoted by~$\tilde{P}$,
        such that
	\begin{equation} \label{ine:strongB2}
	\tilde{P}\psi(x,y) \leq \Gamma(x,y) \, \psi(x,y)
	\end{equation}
	for each $(x,y) \in \X \times \X$.  The following lemma, which
        is a direct corollary of Theorem 1.1 in
        \cite{zhang2000existence}, shows that these two conditions
        are equivalent \citep[see also][Theorem 4.4.3]{kulik2017ergodic}.
	\begin{lemma} \citep[][]{zhang2000existence}
          \label{lem:kernel}
		Suppose that $P_1$ and $P_2$ are Mtks on $(\X, \mathcal{B})$.  Then there
                exists a coupling kernel of $P_1$ and $P_2$, denoted
                by~$\tilde{P}$, such that for each $(x,y) \in \X
                \times \X$,
		\[
		W_{\psi}(\delta_x P_1, \delta_y P_2) = \tilde{P} \psi(x,y) \,.
		\]
	\end{lemma}
	
	It is well-known that, for any $\mu, \nu \in \mathcal{P}(\X)$, there
        exists $\upsilon \in \mathcal{C}(\mu,\nu)$ such that
        $W_{\psi}(\mu, \nu) = \upsilon \psi$ \cite[see, e.g.,][Theorem
          4.1]{villani2008optimal}.  However, taking $\mu = \delta_x
        P_1$ and $\nu = \delta_y P_2$ does \textit{not} trivially
        yield Lemma~\ref{lem:kernel}.  Indeed, the key feature of
        Lemma~\ref{lem:kernel} is that~$\tilde{P}$ is a \textit{bona
          fide} Mtk.
      This is important in
      our analysis as it protects us from potential measurability problems.
      An important consequence of
        Lemma~\ref{lem:kernel} is the convexity of $W_{\psi}$.
	
	\begin{lemma} \label{lem:convex}
		Suppose that $P_1$ and $P_2$ are Mtks on $(\X, \mathcal{B})$.
		Let $\mu, \nu \in \mathcal{P}(\X)$.
		Then, for $\upsilon \in \mathcal{C}(\mu, \nu)$,
		\begin{equation} \label{ine:convex1}
		W_{\psi}(\mu P_1, \nu P_2) \leq \int_{\X \times \X} W_{\psi}(\delta_x P_1, \delta_y P_2) \, \upsilon( \df x, \df y) \,.
		\end{equation}
		Moreover,
		\begin{equation} \nonumber
		W_{\psi}(\mu P_1, \nu) \leq \int_{\X} W_{\psi}(\delta_x P_1, \nu) \, \mu(\df x) \,.
		\end{equation}
	\end{lemma}
	\begin{proof}
%		{\it We may refer to \cite{durmus2015quantitative}, Lemma~1 and omit proof of the first part.}
%		We first show~\eqref{ine:convex1}.
%		By Lemma~\ref{lem:kernel}, there exists a coupling kernel of $P_1$ and $P_2$, denoted by~$\tilde{P}$, such that for each $x,y \in \X$,
%		\[
%		W_{\psi}(\delta_x P_1, \delta_y P_2) = \tilde{P} \psi(x,y) \,.
%		\]
%		This implies that $W_{\psi}(\delta_x P_1, \delta_y P_2)$, as a function of $(x,y)$, is measurable.
%		Let $\upsilon \in \mathcal{C}(\mu, \nu)$, then $\upsilon \tilde{P} \in \mathcal{C}(\mu P, \nu P)$.
%		As a result,
%		\[
%		W_{\psi}(\mu P_1, \nu P_2) \leq \upsilon \tilde{P} \psi = \int_{\X \times \X} W_{\psi}(\delta_x P_1, \delta_y P_2) \, \upsilon( \df x, \df y) \,.
%		\]
		Establishing~\eqref{ine:convex1} using
                Lemma~\ref{lem:kernel} is standard.  See, e.g.,
                \cite{villani2008optimal}, Theorem 4.8.

		Let $P_2$ in~\eqref{ine:convex1} be such that
                $P_2(x,\cdot) = \nu(\cdot)$ for each $x \in \X$, and
                let $\upsilon(\df x, \df y) = \mu(\df x) \nu(\df
                y)$. Then
		\[
		W_{\psi}(\mu P_1, \nu) = W_{\psi}(\mu P_1, \nu P_2)
                \leq \int_{\X \times \X} W_{\psi}(\delta_x P_1, \nu)
                \, \mu(\df x) \nu(\df y) = \int_{\X} W_{\psi}(\delta_x
                P_1, \nu) \, \mu(\df x) \,.
		\]
	\end{proof}
	
%	The next result states that $(B1)$ can be used to construct a bivariate drift condition of a general form.
%	\begin{lemma} \label{lem:bivariate}
%		Suppose that~$P$ is an Mtk on $(\X, \mathcal{B})$, and let~$\tilde{P}$ be a coupling kernel of~$P$.
%		Suppose that $(B1)$ holds.
%		Let $h(x,y) = V(x) + V(y) + 1 \,, x,y \in \X$.
%		Then for each $x,y \in \X$,
%		\[
%		\tilde{P}h(x,y) \leq \frac{PV(x) + PV(y) + 1}{V(x) + V(y) + 1} h(x,y) \,.
%		\]
%	\end{lemma}
%	\begin{proof}
%		Trivial.
%	\end{proof}
	
	The following lemma describes a way of constructing a potential contraction condition based on~\eqref{ine:strongB2} and a ``bivariate"
        drift condition.
	\begin{lemma} \label{lem:strongcontract}
		Suppose that~$P$ is an Mtk on $(\X, \mathcal{B})$ that
                admits a coupling kernel~$\tilde{P}$.  Suppose further
                that there exist measurable functions $h: \X \times \X
                \to [0,\infty)$, $\Lambda: \X \times \X \to
                  [0,\infty)$, and $\Gamma: \X \times \X \to
                    [0,\infty)$ such that for each $(x,y) \in \X
                      \times \X$,
		\[
		\tilde{P} h(x,y) \leq \Lambda(x,y)
                h(x,y) \hspace*{4mm} \mbox{and} \hspace*{4mm}
                \tilde{P} \psi(x,y) \leq \Gamma(x,y) \psi(x,y) \,.
		\]
		For each $r \in (0,1)$, define $\psi_r: \X \times \X
                \rightarrow [0,\infty)$ by $\psi_r(x,y) = \psi(x,y)^r
                  h(x,y)^{1-r}$, and set
                \[
                \rho_r = \sup_{(x,y) \in \X \times \X} \Gamma(x,y)^r
                \Lambda(x,y)^{1-r} \,.
		\]
		Then for every $(x,y) \in \X \times \X$ and $r \in (0,1)$,
		\[
		\tilde{P}\psi_r(x,y) \leq \rho_r \psi_r(x,y) \;.
		\]
	\end{lemma}
	\begin{proof}
		By H\"{o}lder's inequality, for each $r \in (0,1)$ and
                $(x,y) \in \X \times \X$,
		\[
		\tilde{P}\psi_r(x,y) \leq [\tilde{P}\psi(x,y)]^r
                      [\tilde{P}h(x,y)]^{1-r} \leq \Gamma(x,y)^r
                      \Lambda(x,y)^{1-r} \psi_r(x,y) \leq \rho_r
                      \psi_r(x,y) \,.
		\]
	\end{proof}
%	\begin{remark}
%                The fact that $\psi$ is a metric was not used in the
%                proof of Lemma~\ref{lem:strongcontract}.  Thus, the
%                result actually holds for any~$\psi$ that is a
%                non-negative measurable function.
%                In general, it's possible to extend many of the results in this article to the case where $W_{\psi}$ is not induced by a proper metric, although we will not pursue this type of extension for the sake of simplicity.
%	\end{remark}

        In the proof of Lemma~\ref{lem:strongcontract},
        we use H\"{o}lder's inequality to establish a new
        contraction condition.
        \citet{hairer2011asymptotic}, \citet{butkovsky2014subgeometric}, and \citet{douc2018markov}
         make similar use of the inequality.

	We are now ready to prove the main results, namely Theorem~\ref{thm:main} and Propositions~\ref{pro:limiting} and~\ref{pro:stationary}.

	\begin{proof}[Proof of Theorem~\ref{thm:main}]
		By $(B2)$ and Lemma~\ref{lem:kernel}, there exists a
                coupling kernel of~$P$, denoted by~$\tilde{P}$, such
                that for each $(x,y) \in \X \times \X$,
		\[
		\tilde{P} \psi(x,y) \leq \Gamma(x,y) \psi(x,y) \,.
		\]
	        Define $h: \X \times \X \rightarrow [1,\infty)$ as
                  $h(x,y) = V(x) + V(y) + 1$.  For each $(x,y) \in \X
                  \times \X$,
		\[
		\tilde{P} h(x,y) \leq \Lambda(x,y) h(x,y) \,.
		\]
		As in Lemma~\ref{lem:strongcontract}, for $r \in (0,1)$, let $\psi_r: \X \times \X \to [0,\infty)$ be such that
		$
		\psi_r(x,y) = \psi(x,y)^r h(x,y)^{1-r} 
		$.
		Let~$r$ satisfy~\eqref{ine:r2} so that $\rho_r < 1$.  It
                follows from~\eqref{ine:Vpsi} and
                Lemma~\ref{lem:strongcontract} that for each $(x,y)
                \in \X \times \X$ and $n \in \mathbb{Z}_+$,
		\begin{equation} \nonumber
		W_{\psi}(\delta_x P^n, \delta_y P^n) \leq \tilde{P}^n
                \psi(x,y) \leq a^{1-r} \tilde{P}^n \psi_r(x,y) \leq
                a^{1-r} \psi_r(x,y) \rho_r^n \leq a h(x,y) \rho_r^n
                \,.
		\end{equation}
		By Lemma~\ref{lem:convex}, for $\mu,\nu \in \mathcal{P}(\X)$,
		\[
			W_{\psi}(\mu P^n, \nu P^n) \leq a \int_{\X \times \Y} h(x,y) \mu(\df x) \nu(\df y) \rho_r^n = a(\mu V + \nu V + 1) \rho_r^n, \quad n \in \mathbb{Z}_+.
		\]
	\end{proof}

%	We proceed by proving the two propositions on the existence and uniqueness of stationary distributions.
	
	\begin{proof} [Proof of Proposition~\ref{pro:limiting}]
		Fix $x \in \X$.
		By Theorem~\ref{thm:main},
		\begin{equation} \nonumber
		\begin{aligned}
			W_{\psi}(\delta_x P^n, \delta_x P^{n+1}) \leq a (PV(x) + V(x) + 1) \rho_r^n \,.
		\end{aligned}
		\end{equation}
		By (B4), 
		\begin{equation} \label{ine:increment}
			W_{\varrho,p}(\delta_x P^n, \delta_x P^{n+1}) \leq W_{\psi}(\delta_x P^n, \delta_x P^{n+1})^{1/p} \leq  a^{1/p} [PV(x) + V(x) + 1]^{1/p} \rho_r^{n/p} \,.
		\end{equation}
		By the triangle inequality,
                for each positive integer~$n$,
		\[
		\left( \int_{\X} \varrho(x,y)^p P^n(x, \df y) \right)^{1/p} = W_{\varrho,p}(\delta_x, \delta_x P^n) \leq \sum_{k=0}^{n-1} W_{\varrho,p}(\delta_x P^k, \delta_x P^{k+1}) < \infty \,. 
		\]
		Thus, $\delta_x P^n \in \mathcal{P}_{\varrho,p}(\X)$ for each $n \in \mathbb{Z}_+$.
		Moreover,~\eqref{ine:increment} shows that
		\begin{equation} \label{ine:tailsum}
		\sum_{k=n}^{\infty} W_{\varrho,p}(\delta_x P^k, \delta_x P^{k+1}) \leq a^{1/p} [PV(x)+V(x)+1]^{1/p} \frac{\rho_r^{n/p}}{1 - \rho_r^{1/p}}   < \infty \,.
		\end{equation}
		This means that $(\delta_x P^n)_{n \in \mathbb{Z}_+}$ is Cauchy in $(\mathcal{P}_{\varrho,p}(\X), W_{\varrho,p})$.
		Recall that $(\mathcal{P}_{\varrho,p}(\X), W_{\varrho,p})$ is Polish, and thus, complete.
		Hence, there exists $\varpi_x \in \mathcal{P}_{\varrho,p}(\X)$ such that 
		$
		\lim\limits_{n \to \infty} W_{\varrho,p}(\delta_x P^n, \varpi_x) = 0
		$.
		Note that $\varpi_x$ does not depend on~$x$.  To see
                this, let $y \in \X$.  Then by $(B4)$ and Theorem~\ref{thm:main},
		\[
		\lim\limits_{n \to \infty} W_{\varrho,p}(\delta_y P^n, \delta_x P^n)  \leq \lim\limits_{n \to \infty} W_{\psi}(\delta_y P^n, \delta_x P^n)^{1/p} = 0\,.
		\]
		Thus, $\varpi_x = \varpi_y$ for each $(x,y) \in \X \times
                \X$.  We will simply denote $\varpi_x$ by~$\varpi$.
                By~\eqref{ine:tailsum},
        \[
        \begin{aligned}
        	W_{\varrho,p}(\delta_x P^n, \varpi) &\leq \lim\limits_{m \to \infty} \left[ \sum_{k=n}^{m} W_{\varrho,p}(\delta_x P^k, \delta_x P^{k+1}) + W_{\varrho,p}(\delta_x P^{m+1}, \varpi) \right] \\
        	& \leq a^{1/p} [PV(x)+V(x)+1]^{1/p} \frac{\rho_r^{n/p}}{1 - \rho_r^{1/p}} \,.
        \end{aligned}
        \]
%		\[
%		W_{\varrho,p}(\delta_x P^n, \varpi) \leq \sum_{k=n}^{\infty} W_{\varrho,p}(\delta_x P^k, \delta_x P^{k+1}) \leq a^{1/p} [PV(x)+V(x)+1]^{1/p} \frac{\rho_r^{n/p}}{1 - \rho_r^{1/p}} \,.
%		\]
		This gives~\eqref{ine:geometric}.

		We now show that~$P$ has at most one stationary distribution.
		Let $\bar{\varrho} = \varrho \wedge 1$, so that~$\bar{\varrho}$ is a bounded metric that is topologically equivalent to~$\varrho$.
		It follows from~\eqref{ine:geometric} that, for $x \in \X$,
		\[
		W_{\bar{\varrho},p}(\delta_x P^n, \varpi) \leq \left[ a^{1/p} [PV(x)+V(x)+1]^{1/p} \frac{\rho_r^{n/p}}{1 - \rho_r^{1/p}} \right] \wedge 1.
		\]
		Let~$\mu$ be a stationary distribution.
		By Lemma~\ref{lem:convex} (with $\bar{\varrho}(\cdot,\cdot)^p$ as the cost function) and dominated convergence,
                \[
                W_{\bar{\varrho},p}(\mu,\varpi)^p \leq  \lim\limits_{n \to \infty} \int_{\X} W_{\bar{\varrho},p}(\delta_x P^n,\varpi)^p \, \mu(\df x) = \int_{\X} \lim\limits_{n \to \infty} W_{\bar{\varrho},p}(\delta_x P^n,\varpi)^p \, \mu(\df x) = 0.
                \]
        This shows that $\mu = \varpi$. 
	\end{proof}

	\begin{proof} [Proof of Proposition~\ref{pro:stationary}]
		It suffices to show that~$\varpi$ is stationary.

		Assume
                first that $(B5)$ holds.
		It is well-known that $W_{\varrho} = W_{\varrho,1} \leq W_{\varrho,p}$, where $p \geq 1$ is given in $(B4)$ \citep[see, e.g.,][Remark 6.6]{villani2008optimal}.
		Therefore, by Proposition~\ref{pro:limiting},
		\begin{equation} \label{eq:barpsiconverge}
		\lim\limits_{n \to \infty} W_{\varrho}(\delta_x P^n, \varpi) = 0 \,, \quad x \in \X\,.
		\end{equation}
		Fix $x_0 \in \X$.
		For $n \in \mathbb{Z}_+$, let $\upsilon_n \in \mathcal{C}(\delta_{x_0} P^n, \varpi)$ be such that $\upsilon_n \varrho = W_{\varrho}(\delta_{x_0} P^n, \varpi)$.
		For each $n \in \mathbb{Z}_+$, by $(B5)$ and Lemma~\ref{lem:convex},
		\[
		\begin{aligned}
		W_{\varrho}(\varpi, \varpi P) &\leq W_{\varrho}(\varpi, \delta_{x_0} P^{n+1}) + W_{\varrho}(\delta_{x_0} P^{n+1}, \varpi P) \\
		&\leq W_{\varrho}(\varpi, \delta_{x_0} P^{n+1}) + \int_{\X \times \X} W_{\varrho}(\delta_x P, \delta_y P) \,\upsilon_n(\df x, \df y) \\
		& \leq W_{\varrho}(\varpi, \delta_{x_0} P^{n+1}) + K' W_{\varrho}(\delta_{x_0} P^n, \varpi) \,.
		\end{aligned}
		\]
		By~\eqref{eq:barpsiconverge}, letting $n \to \infty$ yields $\varpi = \varpi P$.
		
		Suppose alternatively that~$P$ is Feller.  We know from the proof of
                Proposition~\ref{pro:limiting} that, for any $x \in \X$,
                $(\delta_x P^n)_n$ converges to~$\varpi$ in $(\mathcal{P}_{\varrho,p}(\X),W_{\varrho,p})$,
                which implies that $( \delta_x P^n )_n$ converges
                weakly to~$\varpi$ \cite[see, e.g.,][Theorem
                  6.9]{villani2008optimal}.  Let $f \in C_b(\X)$ be
                arbitrary.  Since~$P$ is Feller, $Pf \in C_b(\X)$, and by weak convergence
                $\lim\limits_{n \to \infty} \delta_x P^n (Pf) = \varpi
                (Pf)$.  It follows that $(\delta_x P^n)_n$ converges
                weakly to $\varpi P$ as well.  This is enough to ensure
                that $\varpi = \varpi P$ \cite[see, e.g.,][Theorem
                  1.2]{billingsley1999convergence}.
	\end{proof}
	
	As promised, we now derive Corollary~\ref{cor:geometric} based on the main theorem.
	\begin{proof}[Proof of Corollary~\ref{cor:geometric}]
		By $(A1)$, condition $(B1)$ holds with the same $V: \X
                \to [0,\infty)$ and $a \in [0,\infty)$.  Recall that $d > 2L/(1-\eta)$, and
                  $C = \big \{ (x,y) \in \X \times \X: V(x) + V(y) < d
                  \big \}$.  By $(A2)$, condition $(B2)$ is satisfied
                  with
		\[
		\Gamma(x,y) = \gamma 1_{(x,y) \in C} + K 1_{(x,y)
                  \not\in C} \,, \quad (x,y) \in \X \times \X\,.
		\]
		Again by $(A1)$, for each $(x,y) \in \X \times \X$,
		\[
		\begin{aligned}
		\frac{PV(x) + PV(y) + 1}{V(x) + V(y) + 1} 
		& \leq \eta + \frac{2L-\eta+1}{V(x) + V(y) + 1} \\
		& \leq \begin{cases}
		2L+1 \,, & (x,y) \in C \,,\\
		\lambda \,, & (x,y) \not\in C \,,
		\end{cases}
		\end{aligned}
		\]
		where $\lambda = (\eta d + 2L + 1)/(d + 1) < 1$.
		Let
		\[
		\Lambda(x,y) = (2L+1) 1_{(x,y) \in C} + \lambda 1_{(x,y) \not\in C} \,, \quad (x,y)
                \in \X \times \X \,.
		\]
		Then
		\[
		\Lambda(x,y) \geq \frac{PV(x) + V(x) + 1}{V(x) + V(y) + 1}
		\]
		for each $(x,y) \in \X \times \X$, as in
                Theorem~\ref{thm:main}.
                
		We now establish $(B3)$ using $(A3)$.
		Let~$r$ satisfy~\eqref{ine:r}, that is,
		\[
		\frac{\log (2L+1)}{\log (2L+1) - \log \gamma} < r < \frac{-\log \lambda}{\log K - \log \lambda} 1_{K > 1} + 1_{K \leq 1} \,.
		\]
		Note that
		\[
		 \sup_{(x,y) \in \X \times \X} \Gamma(x,y)^r \Lambda(x,y)^{1-r} = \left[\gamma^r (2L+1)^{1-r} \right] \vee \left(K^r \lambda^{1-r} \right) \,.
		\]
		If $K \leq 1$, then
		\[
		\left[\gamma^r (2L+1)^{1-r} \right] \vee \left(K^r \lambda^{1-r} \right) \leq \left[\left(\frac{\gamma}{2L+1}\right)^r (2L+1)\right] \vee \lambda^{1-r} < 1 \,.
		\] 
		If, on the other hand, $K>1$, then
		\[
		\left[\gamma^r (2L+1)^{1-r} \right] \vee \left(K^r \lambda^{1-r} \right) = \left[\left(\frac{\gamma}{2L+1}\right)^r (2L+1)\right] \vee \left[\left(\frac{K}{\lambda}\right)^r \lambda\right] < 1 \,.
		\]
		Thus, $(B3)$ holds, and for each~$r$ satisfying~\eqref{ine:r}, $\sup_{(x,y) \in \X \times \X} \Gamma(x,y)^r \Lambda(x,y)^{1-r} < 1$.
		By Theorem~\ref{thm:main}, for $\mu, \nu \in \mathcal{P}(\X)$,
		\begin{equation} \nonumber
		W_{\psi} (\mu P^n, \nu P^n) \leq a (\mu V + \nu V + 1) \rho_r^n \,,
		\end{equation}
		where~$r$ satisfies~\eqref{ine:r}, and
		\[
		\rho_r = \sup_{(x,y) \in \X \times \X} \Gamma(x,y)^r \Lambda(x,y)^{1-r} = \left[\gamma^r (2L+1)^{1-r} \right] \vee \left(K^r \lambda^{1-r} \right) .
		\]
		
		Suppose now that~$P$ has a stationary distribution~$\varpi$.
        By $(A1)$ and \pcite{hairer2018ergodic} Proposition 4.24, $\varpi V \leq L/(1-\eta)$.
        It follows that
         \[
         W_{\psi} (\mu P^n, \varpi) \leq a (\mu V + \varpi V + 1) \rho_r^n \leq a \left(\mu V + \frac{L}{1 - \eta} + 1 \right) \rho_r^n \,.
         \]
	\end{proof}
	
	To end the section, we derive Corollary~\ref{cor:continuous}, which is a continuous analogue of Corollary~\ref{cor:geometric}.

	\begin{proof}[Proof of Corollary~\ref{cor:continuous}]
		To show that~$\varpi$ is the unique stationary distribution of $(P^t)_t$, we use an argument from \cite{butkovsky2014subgeometric}.
		Note that for any $s \geq 0$, $\varpi P^s$ is also stationary for $P^{t_*}$, since $\varpi P^s P^{t_*} = \varpi P^{t_*} P^s = \varpi P^s$.
		Because~$\varpi$ is the unique stationary distribution for $P^{t_*}$, $\varpi P^s = \varpi$ for any $s \geq 0$.
		Hence,~$\varpi$ is the unique stationary distribution for $(P^t)_t$.
		
		By Corollary~\ref{cor:geometric}, for any $\mu \in \mathcal{P}(\X)$ and~$r$ satisfying~\eqref{ine:r},
		\begin{equation} \label{ine:geometricstar}
		W_{\psi}(\mu P^{nt_*}, \varpi) \leq a \left( \mu V + \frac{L}{1-\eta} + 1 \right) \rho_r^n \,.
		\end{equation}
		Let $t \geq 0$, and let $s = t - \lfloor t/t_*
                \rfloor t_* \in [0,t_*)$.  Let $\mu \in
                  \mathcal{P}(\X)$, and let
                  $\upsilon \in \mathcal{C}(\mu P^{\lfloor t/t_*
                    \rfloor t_*} , \varpi)$ be such that $\upsilon \psi =
                  W_{\psi}(\mu P^{\lfloor t/t_* \rfloor t_*}, \varpi)$.
                  Applying Lemma~\ref{lem:convex}
                  and~\eqref{ine:smallt} shows that
		\[
		W_{\psi}(\mu P^t, \varpi) = W_{\psi}(\mu P^{\lfloor t/t_* \rfloor t_*} P^s, \varpi P^s) \leq \int_{\X \times \X} W_{\psi}(\delta_x P^s, \delta_y P^s) \, \upsilon(\df x, \df y) \leq b W_{\psi}(\mu P^{\lfloor t/t_* \rfloor t_*}, \varpi)  \,.
		\]
		The result then follows immediately from~\eqref{ine:geometricstar}.
	\end{proof}

	\section{Examples}
        \label{sec:example}
        
    \subsection{Nonlinear autoregressive process}
	
	Let $(\X,\varrho)$ be the real line equipped with the Euclidean distance, and set $\psi = \varrho$. Let~$P$ be the Mtk of the Markov chain,
        $(X_n)_{n=0}^{\infty}$, defined as follows,
	\[
	X_{n+1} = g(X_n) + Z_n,
	\]
	where $g:\X \to \X$, and $(Z_n)_{n=0}^{\infty}$ is a
        sequence of iid standard normal random variables.  See \cite{douc2004practical} (and the
        references therein) for an in-depth discussion of the
        convergence properties of this family of Markov chains.  In
        this subsection, we compare the numerical bounds resulting from
        applications of Corollary~\ref{cor:geometric} and
        Theorem~\ref{thm:main} to a particular member of the family.

	For illustration, set
	\[
	g(x) = \frac{x}{2} - \frac{\sin x}{2} \,.
	\]
	Then $(X_n)_{n=0}^{\infty}$ is a linear autoregressive chain
        perturbed by a trigonometric term.  We first provide a
        convergence rate bound based on Corollary~\ref{cor:geometric}.
        Letting $V(x) = x^2$ for $x \in \X$, we have
	\begin{equation} \label{ine:driftexm}
	PV(x) = \frac{x^2}{4} - \frac{x \sin x}{2} + \frac{(\sin
          x)^2}{4} + 1 \leq \frac{x^2}{2} + \frac{(\sin x)^2}{2} + 1
        \leq \frac{x^2}{2} + \frac{3}{2} \,.
	\end{equation}
	It follows that $(A1)$ holds with $a=1$, $\eta = 1/2$, and $L
        = 3/2$.  Let $d>2L/(1-\eta) = 6$, and set $C = \big \{ (x,y)
        \in \X \times \X: \, x^2 + y^2 < d \big \}$.  For $(x,y) \in
        \X \times \X$, let
	\begin{equation} \label{eq:gammaexm}
	\begin{aligned}
	\Gamma(x,y) & = \frac{|g(x) - g(y)|}{|x-y|} 1_{x \neq y} + |g'(x)| 1_{x = y}\\ 
	&= \frac{1}{2} \left( 1 - \frac{\sin x - \sin y}{x - y} \right) 1_{x \neq y} + \left( \frac{1}{2} - \frac{\cos x}{2} \right) 1_{x = y} \,.
	\end{aligned}
	\end{equation}
	Then $|g(x)-g(y)| = \Gamma(x,y) |x-y|$.  One can verify that
        $\sup_{(x,y) \in \X \times \X } \Gamma(x,y) = 1$.  Moreover,
        if $d \geq 2 \pi^2$, then $\sup_{(x,y) \in C} \Gamma(x,y) =
        \Gamma(\pi,\pi) = 1$.
	Let $d \in (6,2\pi^2)$, and let $\gamma = \sup_{x^2+y^2 < d}
        \Gamma(x,y)$.  It can be shown that $\gamma < 1$, and $(A2)$
        is satisfied with $K = 1$.  The relation between~$\gamma$ and~$d$ is shown
        in Figure~\ref{fig:gamma}.  Since $K=1$, $(A3)$ holds.  
        
        Remark~\ref{rem:stationary} tells us that~$P$ admits a unique stationary distribution~$\varpi$.
        We can
        now use Corollay~\ref{cor:geometric} to obtain an upper bound on the
        convergence rate of the chain with respect to $W_{\psi} = W_{\varrho}$, namely,
	\[
	\rho_r^A = \left[ \gamma^r (2L+1)^{1-r} \right] \vee \left[
          K^r \left( \frac{\eta d + 2L + 1}{d + 1} \right)^{1-r}
          \right] \,.
	\]
	Note that this bound depends on $r$ and~$d$, both of which can
        be optimized.  The infimum of~$\rho_r^A$ is roughly~$0.976$,
        and this value occurs when $r \approx 0.856$ and $d \approx
        9.2$.  This bound can be improved by letting $V(x) = x^2/c$
        and optimizing $c \in (0,\infty)$, or by finding a sharper
          drift inequality than~\eqref{ine:driftexm}, but we do not
          pursue this any further.

	\begin{figure}
		\centering
		\begin{subfigure}[t]{0.4\textwidth}
			\includegraphics[width=\textwidth]{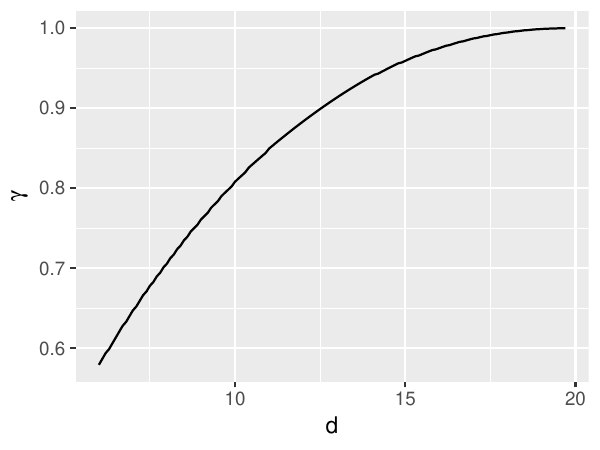}
			\caption{The values of $\gamma =
                          \sup_{x^2+y^2<d} \Gamma(x,y)$ when~$d$
                          ranges over $(6,2\pi^2)$, calculated
                          numerically.} \label{fig:gamma}
		\end{subfigure}
	\hspace{0.5cm}
		\begin{subfigure}[t]{0.4\textwidth}
			\includegraphics[width=\textwidth]{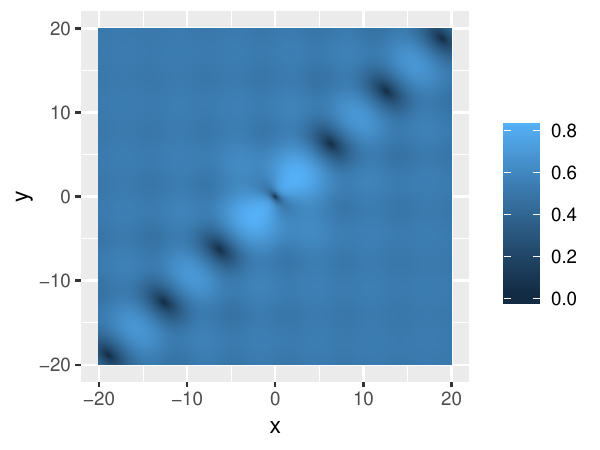}
			\caption{A heat map of $\Gamma(x,y)^r
                          \Lambda(x,y)^{1-r}$ for $r = 0.395$.  The
                          supremum of this function is achieved at
                          around $(-2.3, -2.3)$ and $(2.3, 2.3)$, and
                          is about $0.814$.} \label{fig:heat}
		\end{subfigure}
	\caption{Features of $\Gamma$ and $\Lambda$, defined
          in~\eqref{eq:gammaexm} and~\eqref{eq:lambdaexm}.}
	\end{figure}
	
	We now provide an alternative bound by applying
        Theorem~\ref{thm:main} (or Proposition~\ref{pro:limiting}) directly.  
%        We know from Remark~\ref{rem:piV} that $\varpi V < \infty$, so Theorem~\ref{thm:main} will be able to provide a nontrivial convergence bound if $(B1)$-$(B3)$ hold. 
        By~\eqref{ine:driftexm}, $(B1)$ holds, and for
        every $(x,y) \in \X \times \X$,
	\begin{equation} \label{eq:lambdaexm}
	\frac{PV(x) + PV(y) + 1}{V(x) + V(y) + 1} \leq
        \frac{x^2/2+y^2/2+4}{x^2+y^2+1} =: \Lambda(x,y) \,.
	\end{equation}
	Let $\Gamma: \X \times \X \to [0,\infty)$ be defined as
          in~\eqref{eq:gammaexm}.  
          It's easy to see that $(B2)$ is satisfied.  
          To verify $(B3)$, we will evaluate $\sup_{(x,y) \in \X \times \X} \Gamma(x,y)^r \Lambda(x,y)^{1-r}$ for $r \in (0,1)$.
          Figure~\ref{fig:heat} plots the
          bivariate function $\Gamma(x,y)^r \Lambda(x,y)^{1-r}$ for a
          specific value of $r$.  It is shown in
          Appendix~\ref{app:optimize} that, for each $r \in (0,1)$, one can find 
$
\sup_{(x,y) \in \X \times \X} \Gamma(x,y)^r \Lambda(x,y)^{1-r}
$
         numerically through optimization over a \textit{compact}
         subset of $\X \times \X$.  When~$r$ takes the value $0.395$
         (which is optimal), an upper bound on the chain's convergence rate is
	\[
	\rho_r^B := \sup_{(x,y) \in \X \times \X} \Gamma(x,y)^r
        \Lambda(x,y)^{1-r} \approx 0.814 \,.
	\]
	This is a significant improvement on~$\rho_r^A$, which is
        constructed using standard drift and contraction
        conditions.  The bound can be further improved to $\rho_r^B =
        0.577$ (with $r = 0.382$) if we let
	\begin{equation} \label{eq:lambdaexm2}
	\Lambda(x,y) = \frac{PV(x) + PV(y) + 1}{V(x) + V(y) + 1} =
        \frac{[x/2-\sin(x)/2]^2 + [y/2-\sin(y)/2]^2 + 3}{x^2+y^2+1}
        \,.
	\end{equation}
	We note that it's not really fair to compare the second $\rho_r^B$ with
        $\rho_r^A$, since the latter is based on the loosened drift
        inequality~\eqref{ine:driftexm}.
	
	Of course, in more complicated problems, $\sup_{(x,y) \in \X
          \times \X} \Gamma(x,y)^r \Lambda(x,y)^{1-r}$ would likely be
        much more difficult to evaluate.  Nevertheless, this example
        shows that generalized drift and contraction may contain
        useful information that is not available in their standard
        counterparts.

	\subsection{Gibbs chain for a random effects model}
    
    Let~$q$ and~$m$ be positive integers.
    Consider the random effects model
    \[
    y_{ij} = \mu + \theta_i + e_{ij}, \quad j=1,2,\dots,m, \text{ and }i = 1,2,\dots,q,
    \]
    where, independently, $\theta_i \sim \mbox{N}(0,\lambda_{\theta}^{-1})$, $e_{ij} \sim \mbox{N}(0,\lambda_e^{-1})$.
    Assume that~$\mu$ has a flat prior, and that independently, $\lambda_{\theta}$ and $\lambda_e$ are {\it a priori} $\mbox{Gamma}(a_1,b_1)$ and $\mbox{Gamma}(a_2,b_2)$ distributed, respectively.
    Here, $a_1, a_2, b_1, b_2$ are positive constants, and
    $b_1$ and $b_2$ are rate, rather than scale parameters.
    Denote the observed data $(y_{ij})_{ij}$ by~$y$, and $(\theta_1,\dots,\theta_q)$ by~$\theta$.
    The posterior density function of $(\mu,\theta,\lambda_{\theta},\lambda_e)$ given~$y$ is
    \[
    \begin{aligned}
    	\pi(\mu,\theta,\lambda_{\theta},\lambda_e|y) \propto \, & \lambda_e^{mq/2} \exp \left[ -\frac{\lambda_e}{2} \sum_{i=1}^q \sum_{j=1}^m (y_{ij} - \mu - \theta_i)^2 \right] \\
    	& \times \lambda_{\theta}^{q/2} \exp \left( -\frac{\lambda_{\theta}}{2} \sum_{i=1}^q \theta_i^2 \right) \lambda_{\theta}^{a_1-1} e^{-b_1 \lambda_{\theta}} \lambda_e^{a_2-1} e^{-b_2\lambda_e},
    \end{aligned}
    \]
    where $\lambda_{\theta}, \lambda_e > 0$.
    It can be checked that this posterior is proper.
    
   	Fix $y \in \mathbb{R}^{q  m}$.
    A Gibbs algorithm for sampling from $\pi(\cdot|y)$ can be constructed using the full conditionals $\pi_1(\mu,\theta|\lambda_{\theta},\lambda_e,y)$ and $\pi_2(\lambda_{\theta},\lambda_e|\mu,\theta,y)$, whose exact form can be gleaned from~$\pi(\cdot|y)$ \citep{roman2012convergence}.
    Let $(\mu^{(n)},\theta^{(n)},\lambda_{\theta}^{(n)},\lambda_e^{(n)}), \, n \in \mathbb{Z}_+,$ be the corresponding Markov chain.
    Given $(\mu^{(n)},\theta^{(n)},\lambda_{\theta}^{(n)},\lambda_e^{(n)}) = (\mu,\theta,\lambda_{\theta},\lambda_e)$, the next state, $(\mu^{(n+1)},\theta^{(n+1)},\lambda_{\theta}^{(n+1)},\lambda_e^{(n+1)})$, depends only on $(\mu,\theta)$, and is drawn according to the following steps.
    
    \baro \vspace*{1mm}
    \begin{enumerate}
    	\item Draw $\lambda_{\theta}^{(n+1)}$ from $\mbox{Gamma} \left( q/2+a_1, b_1 + \sum_{i=1}^{p} \theta_i^2/2 \right)$.
    	\item Draw $\lambda_e^{(n+1)}$ from 
    	\[
    	\mbox{Gamma} \bigg( qm/2 + a_2, b_2 + \sum_{i=1}^q \sum_{j=1}^m (y_{ij}-\mu-\theta_i)^2/2 \bigg) .
    	\]
    	\item Draw $\mu^{(n+1)}$ from $\mbox{N} \left( \bar{y}, (\lambda_{\theta}^{(n+1)}+ m\lambda_e^{(n+1)})/(qm\lambda_e^{(n+1)} \lambda_{\theta}^{(n+1)} )  \right)$, where $\bar{y} = (qm)^{-1} \sum_{i,j} y_{ij}$.
    	\item Independently, draw $\theta_i^{(n+1)}$, $i=1,2,\dots,q$, from
    	\[
    	\mbox{N}\left( \frac{m \lambda_e^{(n+1)} (\bar{y}_i - \mu^{(n+1)}) }{\lambda_{\theta}^{(n+1)} + m\lambda_e^{(n+1)}}, \frac{1}{\lambda_{\theta}^{(n+1)} + m\lambda_e^{(n+1)}} \right),
    	\]
    	where $\bar{y}_i = m^{-1} \sum_{j=1}^m y_{ij}$.
    \end{enumerate}
    \vspace*{1mm}
    \barba
    \bigskip
    
	As far as convergence properties are concerned, it suffices to consider the marginal chain $(\mu^{(n)},\theta^{(n)}), \, n \in \mathbb{Z}_+$.
	The state space of this chain is the $(q+1)$-dimensional Euclidean space, which we denote by~$\X$.
	The transition kernel is 
	\[
	P\left((\mu,\theta), \df(\mu',\theta') \right) = \left[ \int_0^{\infty} \int_0^{\infty} \pi_1(\mu',\theta'|\lambda_{\theta},\lambda_e,y) \, \pi_2(\lambda_{\theta},\lambda_e|\mu,\theta,y) \, \df \lambda_{\theta} \, \df \lambda_e \right]  \df \mu' \df \theta' \,,
	\]
	and the stationary distribution~$\varpi$ has density function
	\[
	\frac{\df}{\df \mu \df \theta} \varpi(\mu,\theta) = \int_0^{\infty} \int_0^{\infty}  \pi(\mu,\theta,\lambda_{\theta},\lambda_e|y) \, \df \lambda_{\theta} \, \df \lambda_e \,.
	\]
	Let $\psi = \varrho$ be the metric given by
	\[
	\varrho \left((\mu,\theta), (\mu',\theta') \right) = \sqrt{ q (\mu-\mu')^2 + \sum_{i=1}^q (\theta_i - \theta_i')^2 } \,.
	\]
	\cite{qin2020wasserstein} applied Corollary~\ref{cor:geometric} to study the convergence properties of the chain with respect to $W_{\psi} = W_{\varrho}$.
	Note that \cite{qin2020wasserstein} parametrized $(\mu,\theta)$ differently, so their formula for~$\varrho$ slightly differs from what is given here.
	In Appendix~D of \cite{qin2020wasserstein}, the authors established (up to some minor differences) the following realizations of $(A1)$ and $(B2)$.
	
	\begin{lemma} \label{lem:qinwasser} \citep{qin2020wasserstein}
		For $(\mu,\theta) \in \X$, let
		\[
		V(\mu,\theta) = \frac{1}{q} \sum_{i=1}^q (\theta_i + \bar{y} - \bar{y}_i)^2 + (\mu - \bar{y})^2 \,.
		\]
		Then
		\[
		\varrho\left((\mu,\theta),(\mu',\theta')\right) \leq a \left[ V(\mu,\theta) + V(\mu',\theta') + 1 \right],
		\]
		\[
		PV(\mu,\theta) \leq \eta V(\mu,\theta) + L \,,
		\]
		where $a = 2q$,
		\[
		\eta = \frac{2q + 4}{qm + 2a_2 - 2} + \frac{4}{q + 2a_1 - 2} \,,
		\]
		\[
		L = \frac{(q+2a_1+2) \sum_{i=1}^q (\bar{y}_i - \bar{y})^2 + 4b_1 }{q(q+2a_1-2)} + \frac{q+2}{qm} \frac{\sum_{i=1}^q\sum_{j=1}^m (y_{ij} - \bar{y}_i)^2 + 2b_2 }{qm + 2a_2 - 2} \,.
		\]
		Moreover, for $((\mu,\theta), (\mu',\theta')) \in \X \times \X$,
		\[
		W_{\varrho}(\delta_{(\mu,\theta)} P, \delta_{(\mu',\theta')} P) \leq \Gamma((\mu,\theta),(\mu',\theta')) \, \varrho\left((\mu,\theta), (\mu',\theta') \right) \,,
		\]
		where
		\[
		\begin{aligned}
			\Gamma((\mu,\theta),(\mu',\theta')) =& \left\{ \frac{4}{b_1} \left[ \frac{b_3^2 (q/2+a_1) (q/2+a_1+1) }{ b_1 m^2 (qm/2+a_2-1)(qm/2+a_2-2)  } \wedge 1 \right] \right. \\
			&+ \left. \frac{8(q+2a_1)}{b_1(qm + 2a_2 - 2)} \right\} \sum_{i=1}^q (\bar{y}_i - \bar{y})^2 \\
			&+ \frac{5}{q/2 + a_1 - 1} + \frac{4}{qm/2 + a_2 - 1} \\
			&+ \frac{2b_3 q}{b_1 m (qm + 2a_2- 2)} \wedge \frac{2q}{q+2a_1-2} + \frac{4q}{qm+2a_2-2} \,,
		\end{aligned}
		\]
		\[
			b_3 := b_3((\mu,\theta),(\mu',\theta')) = b_2 + \frac{1}{2} \sum_{i=1}^q \sum_{j=1}^m (y_{ij} - \bar{y}_i)^2 + qm V(\mu,\theta) + qm V(\mu',\theta') \,.
		\]
	\end{lemma}
	
	For the remainder of this subsection, let~$\eta$,~$L$,~$V$, and~$\Gamma$ be defined as in Lemma~\ref{lem:qinwasser}.
	If $\eta < 1$, then $(A1)$ holds.
	To make use of Corollary~\ref{cor:geometric}, one needs to define a coupling set, and find~$\gamma$ and~$K$ in $(A2)$.
	Let the coupling set be of the form
	\[
	C = \left\{ \left((\mu,\theta), (\mu',\theta') \right)  \in \X \times \X: \, V(\mu,\theta) + V(\mu',\theta') < d \right\} ,
	\]
	where $d > 2L/(1-\eta)$.
	Let
	\[
	\gamma = \sup_{((\mu,\theta), (\mu',\theta')) \in C} \Gamma\left((\mu,\theta), (\mu',\theta')\right) , \quad K = \sup_{((\mu,\theta), (\mu',\theta')) \not\in C} \Gamma\left((\mu,\theta), (\mu',\theta')\right).
	\]
	Then $(A2)$ holds with these parameters if $\gamma < 1$ and $K < \infty$.
	Since $\Gamma((\mu,\theta), (\mu',\theta'))$ depends on its argument only through the value of $V(\mu,\theta)+V(\mu',\theta')$, we can let 
	\[
	\Gamma_*(v) = \Gamma\left((\mu,\theta),(\mu',\theta') \right) \text{ if } V(\mu,\theta)+V(\mu',\theta')=v \in [0,\infty) \,.
	\]
	It is easy to obtain the formula of~$\Gamma_*$ from that of~$\Gamma$.
	From there, it is straightforward to establish that $\Gamma_*$ is a non-decreasing function, and that $\Gamma_*(v)$ remains constant when $v \geq v_*$ for some $v_* > 0$ that is determined by $q,m,a_1,a_2,b_1,b_2,$ and~$y$.
	Thus, $\gamma = \Gamma_*(d)$, and $K = \Gamma_*(v_*)$.
	
	In \cite{qin2020wasserstein}, it is shown that, if $m^2$ is a lot larger than $q^{3+\delta}$ for some $\delta > 0$, then, under regularity conditions, one can choose an appropriate~$d$ so that $(A1)$-$(A3)$ are all satisfied with $\eta,L,\gamma,K$ given above.
	The resultant bound via Corollary~\ref{cor:geometric},
	\[
	\rho_r^A = \left[ \gamma^r (2L+1)^{1-r} \right] \vee \left[ K^r \left(
	\frac{\eta d + 2L + 1}{d + 1} \right)^{1-r} \right],
	\]
	is shown to be asymptotically vanishing when $m^2/q^{3+\delta} \to \infty$.
	Although this bound is sharp asymptotically, it can be quite loose for fixed data sets.
	In what follows, we give a numerical illustration, and show how a bound based on Theorem~\ref{thm:main} outperforms $\rho_r^A$.
	
%		\cite{qin2020wasserstein} studied the algorithm's convergence rate for a sequence of data sets indexed by~$q$ under the following conditions:
%		\begin{enumerate}
%			\item[$(S1)$] There exists a constant $\delta > 0$ such that $m^2/q^{3+\delta} \to \infty$ as $q \to \infty$.
%			\item[$(S2)$] There exist positive constants $\ell_1$ and $\ell_2$ such that, for large enough~$q$,
%			\[
%			\frac{1}{q} \sum_{i=1}^q (\bar{y}_i-\bar{y})^2 < \ell_1 \,, \quad \frac{1}{qm} \sum_{i=1}^q \sum_{j=1}^m (y_{ij} - \bar{y}_i)^2 < \ell_2 \,.
%			\]
%		\end{enumerate}
%	One can select appropriate values of $(d,r)$, one for each data set, and use Lemma~\ref{lem:qinwasser} and Corollary~\ref{cor:geometric} to construct a sequence of convergence rate bounds, 
%	\[
%	\rho_r^A = \left[ \gamma^r (2L+1)^{1-r} \right] \vee \left[ K^r \left(
%	\frac{\eta d + 2L + 1}{d + 1} \right)^{1-r} \right].
%	\]
%	It can be shown that, under $(S1)$ and $(S2)$, $\rho_r^A \to 0$ as $q \to \infty$ \citep[][Proposition~25]{qin2020wasserstein}.
%	Although this bound is extremely well-behaved asymptotically, it can be ill-behaved non-asymptotically.
%	In what follows, we give a numerical illustration, and show how a bound based on Theorem~\ref{thm:main} outperforms $\rho_r^A$.

	Let $q=50$ and $m=1000$.
	Set the true value of $(\mu,\lambda_{\theta},\lambda_e)$ to be $(1,10,5)$, and generate~$y$ according to the random effects model.
	In our simulated data set,
	\[
	\frac{1}{q} \sum_{i=1}^q (\bar{y}_i-\bar{y})^2 = 0.0712 \,, \quad \frac{1}{qm} \sum_{i=1}^q \sum_{j=1}^m (y_{ij} - \bar{y}_i)^2 = 0.199 \,.
	\]
	Set $a_1=a_2=b_1=b_2=1$.
	Using formulas in Lemma~\ref{lem:qinwasser}, one can find that $(A1)$ holds with $\eta = 0.0821$ and $L = 0.0787$.
	It can be shown that $\Gamma_*(v)$ is a non-decreasing function that stays constant when $v \geq v_*$, where $v_*$ is a number slightly smaller than~20.
	(See Figure~\ref{fig:Gibbs}.)
	This allows us to evaluate $(\gamma,K)$ for different choices of~$d$ easily.
	Letting $d = 1.35$ and $r=0.176$ yields the sharpest (smallest) bound $\rho_r^A = 0.985$.

	\begin{figure}
		\centering
		\begin{subfigure}[t]{0.4\textwidth}
			\includegraphics[width=\textwidth]{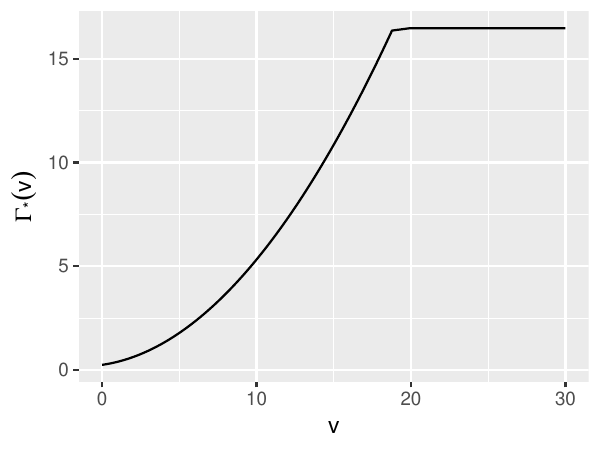}
			%{} \label{fig:Gamma-Gibbs}
		\end{subfigure}
		\hspace{0.5cm}
		\begin{subfigure}[t]{0.4\textwidth}
			\includegraphics[width=\textwidth]{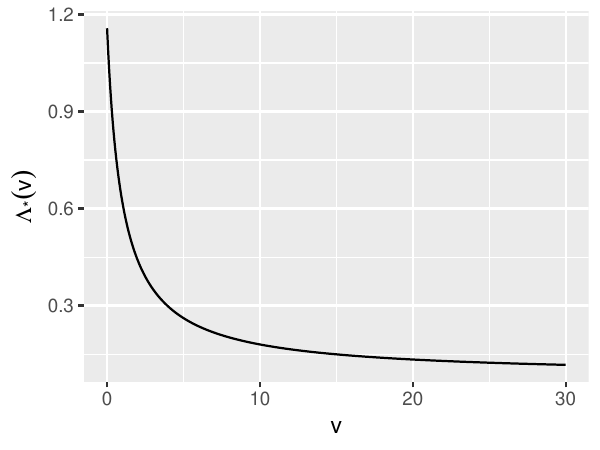}
			%\caption{} \label{fig:Lambda-Gibbs}
		\end{subfigure}
		\caption{Graphs of $\Gamma_*$ and $\Lambda_*$.} \label{fig:Gibbs}
	\end{figure}
	
	We now couple Lemma~\ref{lem:qinwasser} with Theorem~\ref{thm:main} (or Proposition~\ref{pro:limiting}) instead to form a sharper convergence rate bound.
	We see from Lemma~\ref{lem:qinwasser} that $(B1)$ and $(B2)$ hold with the functions~$V$ and~$\Gamma$.
	Let
	\[
	\Lambda((\mu,\theta),(\mu',\theta')) = \frac{\eta [V(\mu,\theta) + V(\mu,\theta)] + 2L + 1 }{V(\mu,\theta) + V(\mu,\theta) + 1}.
	\]
	As with $\Gamma((\mu,\theta),(\mu',\theta'))$, $\Lambda((\mu,\theta),(\mu',\theta'))$ depends on its argument only through the value of $V(\mu,\theta) + V(\mu',\theta')$.
	We can therefore let 
	\[
	\Lambda_*(v) = \Lambda\left( (\mu,\theta),(\mu',\theta') \right) \text{ if } V(\mu,\theta) + V(\mu',\theta') = v \in [0,\infty) \,.
	\]
	$\Lambda_*(v) = (\eta v + 2L +1)/(v+1)$ is decreasing in~$v$.
	(See Figure~\ref{fig:Gibbs}.)
	Moreover, recall that $\Gamma_*(v)$ is non-decreasing and constant when $v \geq 20$.
	For $r \in (0,1)$, let $\rho_r^B$ be the convergence rate bound constructed via Theorem~\ref{thm:main}, so that
	\[
	\begin{aligned}
		\rho_r^B &= \sup_{((\mu,\theta),(\mu',\theta'))} \Gamma\left( (\mu,\theta),(\mu',\theta') \right)^r \Lambda\left( (\mu,\theta),(\mu',\theta') \right)^{1-r} \\
		&= \sup_{v \in [0,20]} \Gamma_*(v)^r \Lambda_*(v)^{1-r} \,.
	\end{aligned}
	\]
	Let $r=0.335$.
	Then $\Gamma_*(v)^r \Lambda_*(v)^{1-r}$ is maximized when $v=18.78$, and this yields $\rho_r^B = 0.687$, which is significantly smaller than $\rho_r^A = 0.985$.

	\vspace{0.5cm}
	\noindent {\bf \large Acknowledgment}.  We thank the Editor, the Associate Editor, and two anonymous
        reviewers for helpful comments and suggestions.  The second
        author was supported by NSF Grant DMS-15-11945.
	
	\vspace{1cm}
	
	\noindent{\bf \Large Appendix}
	\appendix
	
        \section{Comparison of $\rho_r$ and Durmus and Moulines's Geometric Rate Bound} 
        \label{app:compare}

We begin with a formal statement of Durmus and Moulines's result.

	\begin{proposition} 
          \label{pro:durmus} \citep[][]{durmus2015quantitative}
		Suppose that the metric~$\varrho$ is bounded above
                by~$1$, and that the following two conditions hold:
		\begin{enumerate}
			\item [$(C1)$] There exists a measurable
                          function $\bar{V}: \X \to [1,\infty)$, $\eta'
                            \in [0,1)$, and $L' \in \mathbb{R}$ such
                                that
			\[
			P\bar{V}(x) \leq \eta' \bar{V}(x) + L' \,, \quad x \in \X \,.
			\]
			\item [$(C2)$] There exists $\gamma' \in [0,1)$ such
                          that for each $(x,y) \in \X \times \X$,
			\[
			W_{\varrho}(\delta_x P, \delta_y P) \leq \begin{cases}
			\gamma' \varrho(x,y) \,, & (x,y) \in C' \,,\\
		        \varrho(x,y) \,, & (x,y) \not\in C' \,,
			\end{cases}
			\]
			where the coupling set~$C'$ is defined to be
                        $\big \{ (x,y) \in \X \times \X :\, \bar{V}(x)
                        + \bar{V}(y) \leq (2L'+\delta')/(1-\eta') \big
                        \}$, for some $\delta' > 0$ \,.
		\end{enumerate}
	Then~$P$ admits a unique stationary distribution~$\varpi$.
        Moreover, for each $x \in \X$ and every $n \in \mathbb{Z}_+$,
	\[
	W_{\varrho}(\delta_x P^n, \varpi) \leq \left( \frac{1}{2} +
        \frac{1}{(\gamma')^2} \right) \left( 1 + \frac{L'}{1-\eta'} \right)
        V(x) \left(\rho^{\mbox{\scriptsize{DM}}} \right)^n \,,
	\]
	where
	\[
	\rho^{\mbox{\scriptsize{DM}}} = \exp \left( -\frac{\log
          (1/\lambda) \log (1/\gamma')}{\log J + \log (1/\gamma')} \right) \,,
	\]
	$\lambda = 2L'(1-\eta')/(2L'+\delta') + \eta'$, and $J =
        (2L'+\delta')/(1-\eta') + 2L'/\lambda \,$.
	\end{proposition}

The following result shows that
Corollary~\ref{cor:geometric} can always be used to improve
upon $\rho^{\mbox{\scriptsize{DM}}}$.

\begin{proposition}
  \label{prop:better}
    Assume that~$\varrho$ is bounded, and $(C1)$ and $(C2)$ are satisfied.  Then
    Corollary~\ref{cor:geometric} can be used to construct an
    alternative convergence rate bound, $\rho_r$, such that $\rho_r <
    \rho^{\mbox{\scriptsize{DM}}}$.
\end{proposition}

\begin{proof}
Since $\bar{V} \geq 1$, $(C1)$ implies that $L' \geq (1-\eta') \varpi
\bar{V} \geq 1 - \eta' > 0$ \citep[see][Proposition
  4.24]{hairer2018ergodic}.  Thus, $\lambda \in (\eta', 1)$ and $J >
1$.

We now translate $(C1)$ and $(C2)$ into $(A1)$-$(A3)$ and apply
Corollary~\ref{cor:geometric}.  Take $\psi = \varrho$ and $V =
\bar{V} - 1/2$.  Then by $(C1)$ and the boundedness of~$\psi$, $(A1)$ holds with
$a = 1$, $\eta = \eta'$, and $L = L' + \eta'/2 - 1/2$.  (Note that $L
\geq 0$.)  Take $d = (2L + \delta')/(1-\eta)$, and set
\[
C = \big\{ (x,y) \in \X \times \X : \, V(x) + V(y) < d \big\} =
\bigg\{ (x,y) \in \X \times \X : \, \bar{V}(x) + \bar{V}(y) <
\frac{2L'+\delta'}{1-\eta'} \bigg\} \subset C' \,.
\]
By $(C2)$, condition $(A2)$ holds with $\gamma = \gamma'$, $K = 1$,
and~$C$ given above.  
Since $K=1$, $(A3)$ holds.
Corollary~\ref{cor:geometric}, along with Remark~\ref{rem:stationary}, states that the chain's $\varrho$-induced
Wasserstein distance to its unique stationary distribution decreases
at a geometric rate of (at most)
\[
\rho_r = \big[ \gamma^r(2L+1)^{1-r} \big] \vee \left( \frac{\eta d +
  2L + 1}{d + 1} \right)^{1-r} \,,
\]
where
\[
r \in \bigg( \frac{\log(2L+1)}{\log(2L+1)+\log (1/\gamma)}, 1 \bigg) \;.
\]
Note that
\[
\frac{\eta d + 2L + 1}{d + 1} = \eta' + \frac{2L'}{d + 1} = \lambda
\,.
\]
Now set
\[
r = \frac{\log (2L+1) + \log (1/\lambda)}{\log (2L+1) + \log (1/\lambda) +
  \log (1/\gamma)} \;.
\]
A direct calculation shows that, with this value of $r$, we have
\[
\gamma^r(2L+1)^{1-r} = \lambda^{1-r} = \exp \left( -\frac{\log (1/\lambda) \log (1/\gamma')}
{\log (2L+1) + \log (1/\lambda) + \log (1/\gamma')} \right) .
\]
It follows that
\[
\rho_r = \exp \left( -\frac{\log (1/\lambda) \log (1/\gamma')}
{\log (2L+1) + \log (1/\lambda) + \log (1/\gamma')} \right) .
\]
Since $L' \geq 1 - \eta'$, one has $J\lambda > 2\lambda + 2L' > \eta' + 2L' =
2L+1$.  As a result, $0 < \log (2L+1) + \log (1/\lambda) < \log J$, which
implies that $\rho_r < \rho^{\mbox{\scriptsize{DM}}}$.
\end{proof}

Finally, as argued at the end of Section~\ref{sec:general}, if
$\rho_r$ is calculated based on a set of generalized drift and
contraction conditions, then it may be further improved.  As
demonstrated in Section~\ref{sec:example}, the convergence rate bound
in Theorem~\ref{thm:main} can be considerably shaper than that in
Corollary~\ref{cor:geometric}, and thus, substantially sharper than
$\rho^{\mbox{\scriptsize{DM}}}$ as well.

	\section{Finding~$\rho_r^B$ for the Perturbed Linear Autoregressive Chain}
        \label{app:optimize}

	Let $\X = \mathbb{R}$, and consider the linear autoregressive
        chain perturbed by a trigonometric term from
        Section~\ref{sec:example}.  We now explain how to find
	\[
	\rho_r^B = \sup_{(x,y) \in \X \times \X} \Gamma(x,y)^r
        \Lambda(x,y)^{1-r} \,, \quad r \in (0,1) \,,
	\]
	where~$\Lambda$ and~$\Gamma$ are, respectively, given by~\eqref{eq:lambdaexm} and~\eqref{eq:gammaexm}.
	The main result is as follows.
	
	\begin{proposition} \label{pro:app1}
		Let~$\Lambda$ and~$\Gamma$ be defined as in~\eqref{eq:lambdaexm} and~\eqref{eq:gammaexm}.
		Then for any $r \in  (0,1)$,
		\[
		\arg\max_{(x,y) \in \X \times \X} \Gamma(x,y)^r
                \Lambda(x,y)^{1-r} \subset C_0 := \{(x,y): \, x^2 +
                y^2 \leq 2\pi^2 \} \,.
		\]
	\end{proposition}
	
	\begin{proof}
		Let $r \in (0,1)$ and $(x,y) \in (\X \times \X) \setminus C_0$ be arbitrary.
		It suffices to show that there exists $(x',y') \in C_0$ such that
		\begin{equation} \label{ine:app1}
		\Gamma(x,y)^r \Lambda(x,y)^{1-r} < \Gamma(x',y')^r
                \Lambda(x',y')^{1-r} \,.
		\end{equation}
		By the mean value theorem and periodicity, there exists a point $\xi
                \in [x,y]$ (or $[y,x]$), as well as a point $\xi' \in
                    [-\pi,\pi]$ such that
		\[
		\Gamma(x,y) = \frac{1}{2}- \frac{\cos \xi}{2} =
                \frac{1}{2}- \frac{\cos \xi'}{2} = \Gamma(\xi',\xi')
                \,.
		\]
		Note that $(\xi',\xi') \in C_0$.  Moreover, it's easy
                to verify that $\Lambda(x,y) < \Lambda(\xi',\xi')$.
                As a result,~\eqref{ine:app1} holds with $x'=y'=\xi'$.
	\end{proof}
	
	Proposition~\ref{pro:app1} implies that, to maximize
        $\Gamma(x,y)^r \Lambda(x,y)^{1-r}$ over $\X \times \X$, we
        only need to restrict our attention to the compact set $C_0$.
        Since the objective function is uniformly continuous on $C_0$,
        we can solve the problem by optimizing over a sufficiently
        fine grid.
	
	Finally, assume instead that~$\Lambda$ is given
        by~\eqref{eq:lambdaexm2}.  The analog of
        Proposition~\ref{pro:app1} is as follows.
	\begin{proposition}
		Let~$\Lambda$ and~$\Gamma$ be defined as
                in~\eqref{eq:lambdaexm2} and~\eqref{eq:gammaexm}.
                Then for any $r \in (0,1)$,
		\[
		\arg \max_{(x,y) \in \X \times \X} \Gamma(x,y)^r
                \Lambda(x,y)^{1-r} \subset C'_0 := \{(x,y): \, |x|\leq
                26, \, |y| \leq 26 \} \,.
		\]
	\end{proposition}
	\begin{proof}
		Let $r \in (0,1)$ and $(x,y) \in (\X \times \X) \setminus C'_0$ be arbitrary.
		It suffices to show that there exists $(x',y') \in C'_0$ such that~\eqref{ine:app1} holds.
		As in the proof of Proposition~\ref{pro:app1}, there exists a point $\xi \in [-2\pi,-\pi]\cup[\pi,2\pi]$ such that $\Gamma(x,y) = \Gamma(\xi,\xi)$.
		Note that $(\xi,\xi) \in C'_0$.
		Moreover,~$\xi$ and $\sin \xi$ have opposite signs.
		Thus,
		\[
		\Lambda(\xi,\xi) \geq \frac{\xi^2/2+3}{2\xi^2+1} > 0.284 \,.
		\]
		Since $(x,y) \not\in C'_0$, we have $|x| > 26$ or $|y| > 26$.
		Without loss of generality, assume that the former holds.
		Then $|x| < x^2/26$, and $|y| + (x^2+1)/|y| > 52$.
		It follows that
		\[
		\begin{aligned}
		\Lambda(x,y) &\leq \frac{(|x|+1)^2/4 + (|y|+1)^2/4 + 3}{x^2 + y^2 + 1} \\
		& = 0.25 + \frac{0.5|x|+0.5|y|+3.25}{x^2+y^2+1} \\
		& < 0.25 + \frac{x^2}{52(x^2+y^2+1)} + \frac{0.5}{|y| + (x^2+1)/|y|} + \frac{3.25}{x^2+y^2+1} \\
		& < 0.25 + 1/52 + 0.5/52 + 3.25/(26^2+1) \\
		& < 0.284 \,. 
		\end{aligned}
		\]
		Hence, $\Lambda(x,y) < \Lambda(\xi,\xi)$, and~\eqref{ine:app1} holds with $(x',y') = (\xi,\xi)$.
	\end{proof}

	\bibliographystyle{ims} 
	\bibliography{qinbib-wasserdrift}
\end{document}